\documentclass[11pt]{amsart} % [...,a4paper]
\usepackage{amssymb, amsthm, amsfonts, url} % ,a4wide
\newtheorem{lemma}{\bf Lemma}[section]

\newtheorem{corollary}[lemma]{\bf Corollary}
\newtheorem{proposition}[lemma]{\bf Proposition}
\newtheorem{fact}[lemma]{\bf Fact}
\newtheorem{definition}[lemma]{\bf Definition}

\newcommand{\Fil}{{\mathsf{\Phi}}}

\begin{document}

%...... title
\title{Order and interval topologies on complete Boolean algebras}

%...... authors
\author{Dominic van der Zypen}
\address{Federal Office of Social Insurance, CH-3003 Bern,
Switzerland}
\email{\tt dominic.zypen@gmail.com}

%......MSC subject class

%...... Abstract...
\begin{abstract} 
Problem 76 of Birkhoff's {\em Lattice Theory} \cite{Bir} asks whether
for complete Boolean algebras the order topology and the interval
topology coincide. We answer this question in the negative.
\end{abstract}

%....... main()
\maketitle
\parindent = 0mm
\parskip = 2 mm
% . . . . . . . . . . . . . . . . .
\section{Introduction}

On p.~166 of Birkhoff's {\em Lattice Theory} \cite{Bir}, the following
problem is stated:
\begin{quote}
Does the order topology and the interval topology coincide 
for a complete Boolean algebra?
\end{quote}

In the following we introduce the concepts necessary
to understand and answer the question.

A {\em partially ordered set} (or {\em poset} for short) 
is a set $X$ with a binary
relation $\leq$ that is reflexive, transitive, and
anti-symmetric (i.e., $x,y\in X$ with $x\leq y$ and $y\leq x$
implies $x=y$). Often, a poset is denoted by $(X,\leq)$. 
A subset $D\subseteq X$ is called a {\em down-set} if it is
``closed under going down'', that is $d\in D, x\in X, x\leq d$
jointly imply $x\in D$. A special case of a down-set
is the set $$\downarrow_P x = \{y\in X: y\leq x\}$$
for $x\in X$. (Sometimes we just write $\downarrow x$
if the poset $P$ is clear from the context.)
Down-sets of this form are called {\em principal}.
If $S\subseteq X$ we say $S$ has a {\em smallest element} $s_0\in S$
if $s_0\leq s$ for all $s\in S$. Note that anti-symmetry
of $\leq$ implies that a smallest element is unique (if it exists
at all!). Similarly, we define a largest element. Moreover,
we set $$S^u = \{x\in X: x\geq s \textrm{ for all } s \in S\}$$
to be the {\em set of upper bounds} of $S$. The set
of lower bounds $S^{\ell}$ is defined analogously.

We say that a subset $S\subseteq X$ of a poset $(X,\leq)$
has an {\em infimum} or {\em largest lower bound} if
\begin{enumerate}
\item $S^{\ell} \neq \emptyset$, and
\item $S^{\ell}$ has a largest element.
\end{enumerate}
Again, an infimum (if it exists) is unique by anti-symmetry of the 
ordering relation, and it is
denoted by $\textrm{inf}(S)$ or $\bigwedge_X S$. The dual
notion (everything taken ``upside down'' in the poset)
is called {\em supremum} and is denoted by $\textrm{sup}(S)$ 
or $\bigvee_X S$. The infimum of the empty set is defined
to be the largest element of $X$ if it has one, and
the supremum is the smallest element of $X$. 

A poset $(X,\leq)$
in which infima and suprema exist for all $S\subseteq X$ 
is called a {\em complete lattice}. A {\em lattice}
has suprema and infima for finite non-empty subsets.
If $(X,\leq)$ is a poset and $x,y\in X$ we use the 
following notation $$x\vee y := \bigvee_X\{x,y\},$$ and
$x\wedge y$ is defined analogously. To emphasize the
binary operations $\vee,\wedge$, a lattice $(L,\leq)$
is sometimes written as $(L,\vee,\wedge)$. A lattice
is {\em distributive} if for all $x,y,z\in L$ we have
$$x\wedge(y\vee z) = (x\vee y)\wedge(x\vee z).$$

\begin{definition}
Given a poset $(X,\leq)$, the {\em interval topology} 
$\tau_i(X)$ is
given by the subbase $${\mathcal S} = \{X\setminus 
(\downarrow x): x\in X\} \cup \{X\setminus (\uparrow x):
x\in X\}.$$
\end{definition}

%-------------------------
\section{Convergence spaces}
We need the notion of 
order convergence expressed with filters (Birkhoff uses nets,
and filters offer an equivalent, but more concise approach to convergence
\cite{NetFil}). The underlying concept is that of a convergence
space.

Let $X\neq \emptyset$ be a set. By a {\em set filter} 
$\mathcal{F}$ on $P$ we mean 
a collection of subsets of $P$ such that:
\begin{itemize}
\item[-] $\emptyset \notin \mathcal{F}$;
\item[-] $A, B\in \mathcal{F}$ implies $A\cap B\in \mathcal{F}$;
\item[-] $U\in \mathcal{F}$, $U'\subseteq P$ and $U'\supseteq U$  
implies $U'\in \mathcal{F}$.
\end{itemize}

If $\mathcal{B}$ is a collection of subsets of $X$ such that
\begin{itemize}
\item[-] $\emptyset\notin \mathcal{B}$,
\item[-] for $A, B\in\mathcal{B}$ there is $C\in\mathcal{B}$
with $C\subseteq A\cap B$,
\end{itemize}
then we call $\mathcal{B}$ a {\em filter base}. The {\em
filter generated by} $\mathcal{B}$ is the collection
of subsets of $X$  that contain some member of $\mathcal{B}$.

If $\mathcal{F}\subseteq\mathcal{G}$ are filters on $X$
we say that $\mathcal{G}$ is a 
{\em super-filter} of $\mathcal{F}$. An {\em ultrafilter}
is a maximal filter with respect to set inclusion (i.e.
it has no proper super-filters). Given $x_0\in X$ we set
$$P_{x_0} =\{A\subseteq X: x_0\in A\},$$ which is easily
seen to be an ultrafilter.

By $\Fil(X)$ we denote the set of filters on $X$.

\begin{definition}\label{convdef}
A {\em convergence space} is a pair $(X,\to)$ where
$X$ is a non-empty set and $\to\subseteq \Fil(X)\times X$ is
a relation satisfying the following properties:
\begin{enumerate}
\item If $\mathcal{F}\subseteq\mathcal{G}$ are elements
of $\Fil(X)$ such that $\mathcal{F}\to x$, then $\mathcal{G}\to x$, and
\item for all $x\in X$ we have $P_x \to x$.
\end{enumerate}

(Note that we write $\mathcal{F} \to x$ instead of $(\mathcal{F}, x)\in\to$.)

If $\mathcal{F}\to x$ for some $\mathcal{F}\in\Fil(X)$ and $x\in X$ we 
say that $\mathcal{F}$ {\em converges} to $x$.
\end{definition}
To every convergence relation $\to$ as described above we
can associate a topology on the base set $X$ by setting
$$\tau_{\to} = \{U\subseteq X: \textrm{ if } \mathcal{F}\in\Fil(X)
\textrm{ and } u\in U \textrm{ with } \mathcal{F}\to u \textrm{ we have }
U\in\mathcal{F}\}.$$
It is a routine exercise to verify that $\tau_\to$ is a topology.

We will use the following fact later on:
\begin{fact}\label{convfact}
If a filter $\mathcal{F}$ contains all open neighborhoods of $x$
in the topological space $(X,\tau_\to)$ then $\mathcal{F}\to x$.
\end{fact}
Interestingly, many topological properties such as compactness,
Hausdorffness, and more can be put in terms of convergence
spaces. We will need the following later on:
\begin{proposition}\label{hausdorfflem}
If $(X,\to)$ is a convergence space, then the following are equivalent:
\begin{enumerate}
\item every $\mathcal{F}\in\Fil(X)$ converges to at most one point $x\in X$;
\item $(X,\tau_\to)$ is Hausdorff.
\end{enumerate}
\end{proposition}
\begin{proof}
(1) $\Rightarrow$ (2). Assume that $(X,\tau_\to$ is not Hausdorff. Then
there are $x\neq y$ such that every neighborhood $U$ of $x$ intersects
every neighborhood $V$ of $y$. We let $\mathcal{F}$ be the collection
of all $U\cap V$ where $U$ is a neighborhood of $x$ and $V$ is a neighborhood
of $y$. A routine verification shows that $\mathcal{F}$ is indeed a filter,
and Fact \ref{convfact} implies that 
both $\mathcal{F}\to x$ and $\mathcal{F}\to y$,
so $\mathcal{F}$ does not converge to a unique point.

(2) $\Rightarrow$ (1). To be done.
\end{proof}
%-------------------------
\section{Order convergence}
Let $(P,\leq)$ be a poset.

If $\mathcal{F}$ is a set filter on $P$, then we set ${\mathcal F}^u = 
\bigcup\{F^u: F\in \mathcal{F}\}$ and define ${\mathcal F}^\ell$ similarly. 
For $x\in P$ and ${\mathcal F}$ a set filter on $P$ we write 
$${\mathcal F}\to_o x \textrm{ iff } \bigwedge\mathcal{F}^u = x = 
\bigvee \mathcal{F}^\ell$$ 
and say ${\mathcal F}$ {\em order-converges} to $x$.

It is easy to verify that the order convergence relation is
indeed a convergence relation as described in definition \ref{convdef}.
The topology $\tau_{\to_{o}}$ associated to $\to_o$  on the poset $(P,\leq)$
is denoted by $\tau_o(X)$, and we call it the 
{\em order (convergence) topology}.

Lemma \ref{hausdorfflem} and the fact that every filter on
a poset order-converges to at most one point by definition
jointly imply the following:
\begin{corollary}
If $(P,\leq)$ is a poset, then $(P,\tau_o(P))$ is Hausdorff.
\end{corollary}

%-------------------------
\section{Interval topology on complete Boolean algebras}
\begin{proposition}\label{prop_keith}
The interval topology of a complete atomless Boolean 
algebra (such as $\mathcal{P}(\omega)/\mathrm{fin}$)
is not Hausdorff.
\end{proposition}
\begin{proof}
The set of principal ideals and principal 
filters form a subbasis of closed
sets for the topology on $B$, so a typical basic closed set has the form
$I\cup F$ where $I$ is a finitely generated order ideal and $F$ is a finitely
generated order filter. Our goal is to show that there do not
exist proper basic closed subsets $C_0=I_0\cup F_0$ containing $0$ and not $1$
and $C_1=I_1\cup F_1$ containing $1$ and not $0$,
whose union $C_0\cup C_1$ equals $B$.
To obtain a contradiction, assume that
$B = C_0\cup C_1$ for proper basic closed sets
satisfying $0\in C_0-C_1$ and $1\in C_1-C_0$.

Since $1\notin C_1$ and $0\notin C_0$ we get that
$C_0 = I_0$ is a proper f.g. order ideal and $C_1 = F_1$ is a
proper f.g. order filter.

Suppose that $C_0$ is the order ideal generated by
some finite set $X$ and that $C_1$ is the order filter
generated by some finite set $Y$.
Let $C$ be the subalgebra of $B$ generated by $X\cup Y$. It is finite.
If the set of atoms of $C$ is $V=\{v_1\ldots,v_p\}$, then
for $u_i:=\bigvee_{j\neq i} v_j$ we get that
$U = \{u_1,\ldots,u_p\}$ is the set of coatoms of $C$.
Note that $V$ consists of the cells in a finite partition
of unity, and $U$ consists of the complementary elements.

The f.g. order ideal $D_0$ generated by
$U$ contains $C_0$ and still does not contain $1$, while the f.g.
order filter $D_1$ generated by $V$ contains $C_1$ but still does
not contain $0$. This reduces the original situation to one in
which $D_1$ is a proper order filter generated by the
cells in a finite partition of unity, while $D_0$ is the proper f.g.
order ideal generated by the complements of single cells
in the same partition of unity.

The set $V = \{v_1,\ldots,v_p\}$ consists of nonempty,
pairwise disjoint, elements of $B$. If $B$ is atomless, 
then we can choose $w_i$ such that $0<w_i<v_i$ for all $i$.
Let $W = \bigvee w_i$.
For each $i$ we have $v_i\not\leq W$,
since $v_i\cap W = v_i\cap w_i = w_i < v_i$. Moreover,
$W\not\leq u_j$ for any $j$, since $0<w_j\leq W$
and $w_j\not\leq u_j$. Thus $W\notin D_0\cup D_1$.
Hence also $W\notin C_0\cup C_1$.

In fact, the argument can be relativized to intervals to prove
that if $x<y$, then $x$ cannot be separated from $y$ unless
there is an atom below $y$ disjoint from $x$. This happens for every 
$x<y$ only when $B$ is an
atomic Boolean algebra.
To restate this: if the topology is Hausdorff, then the Boolean
algebra must be atomic.
\end{proof}

As the order topology on any poset is Hausdorff, 
proposition \ref{prop_keith} provides a negative answer
to Birkhoff's question.

\begin{corollary}
On $\mathcal{P}(\omega)/\mathrm{fin}$ the order topology
and the interval topology do not agree.
\end{corollary}
Note that on $\mathcal{P}(\omega)/\mathrm{fin}$ is
the order topology is strictly finer than the interval
topology, since on any poset, the order topology contains
the interval topology, see \cite{DZ}.
\section{Acknowledgement}
I want to thank Keith A.~Kearnes (University of Colorado, Boulder)
for his help on Proposition \ref{prop_keith} \cite{Kear}.
%................... bibliography .................
%-- end footnotesize
\end{document}